\documentclass[11pt,english]{article}
\usepackage[T1]{fontenc}
\usepackage{textcomp}
\usepackage[latin9]{inputenc}
\usepackage{refstyle}
\usepackage{amsmath}
\usepackage{amsthm}
\usepackage{amsfonts}
\usepackage{amssymb}
\usepackage{csquotes}
\usepackage{setspace}
\makeatletter
\ifx\proof\undefined
\newenvironment{proof}[1][\protect\proofname]{\par
	\normalfont\topsep6\p@\@plus6\p@\relax
	\trivlist
	\itemindent\parindent
	\item[\hskip\labelsep\scshape #1]\ignorespaces
}{%
	\endtrivlist\@endpefalse
}
\providecommand{\proofname}{Proof}
\fi

\makeatother

\theoremstyle{plain}
\newtheorem{thm}{\protect\theoremname}
\theoremstyle{definition}
\newtheorem{defn}[thm]{\protect\definitionname}
\theoremstyle{plain}
\newtheorem{prop}[thm]{\protect\propositionname}
\theoremstyle{remark}
\newtheorem{rem}[thm]{\protect\remarkname}
\theoremstyle{plain}
\newtheorem{lem}[thm]{\protect\lemmaname}
\usepackage{babel}
\usepackage[style=authoryear]{biblatex}
\newref{thm}{
        name      = \RSthmtxt,
        names     = \RSthmstxt,
        Name      = \RSThmtxt,
        Names     = \RSThmstxt,
        rngtxt    = \RSrngtxt,
        lsttwotxt = \RSlsttwotxt,
        lsttxt    = \RSlsttxt
}

\newref{def}{
        name      = \RSdefntxt,
        names     = \RSdefnstxt,
        Name      = \RSDefntxt,
        Names     = \RSDefnstxt,
        rngtxt    = \RSrngtxt,
        lsttwotxt = \RSlsttwotxt,
        lsttxt    = \RSlsttxt
}

\newref{prop}{
        name      = \RSproptxt,
        names     = \RSpropstxt,
        Name      = \RSProptxt,
        Names     = \RSPropstxt,
        rngtxt    = \RSrngtxt,
        lsttwotxt = \RSlsttwotxt,
        lsttxt    = \RSlsttxt
}

\newref{rem}{
        name      = \RSremtxt,
        names     = \RSremstxt,
        Name      = \RSRemtxt,
        Names     = \RSRemstxt,
        rngtxt    = \RSrngtxt,
        lsttwotxt = \RSlsttwotxt,
        lsttxt    = \RSlsttxt
}

\newref{lem}{
        name      = \RSlemtxt,
        names     = \RSlemstxt,
        Name      = \RSLemtxt,
        Names     = \RSLemstxt,
        rngtxt    = \RSrngtxt,
        lsttwotxt = \RSlsttwotxt,
        lsttxt    = \RSlsttxt
}

\def\RSdefntxt{definition~}
\def\RSdefnstxt{definitions~}
\def\RSDefntxt{Definition~}
\def\RSDefnstxt{Definitions~}
\def\RSlemtxt{lemma~}
\def\RSlemstxt{lemmas~}
\def\RSLemtxt{Lemma~}
\def\RSLemstxt{Lemmas~}
\def\RSproptxt{proposition~}
\def\RSpropstxt{propositions~}
\def\RSProptxt{Proposition~}
\def\RSPropstxt{Propositions~}
\def\RSremtxt{remark~}
\def\RSremstxt{remarks~}
\def\RSRemtxt{Remark~}
\def\RSRemstxt{Remarks~}
\def\RSthmtxt{theorem~}
\def\RSthmstxt{theorems~}
\def\RSThmtxt{Theorem~}
\def\RSThmstxt{Theorems~}
\providecommand{\definitionname}{Definition}
\providecommand{\lemmaname}{Lemma}
\providecommand{\propositionname}{Proposition}
\providecommand{\remarkname}{Remark}
\providecommand{\theoremname}{Theorem}

\begin{document}
\title{Rates of forgetting for the sequentially Markov coalescent}
\author{Jonathan Terhorst\\Department of Statistics\\University of Michigan\footnote{Corresponding author: \texttt{jonth@umich.edu}}}
\date{\today}
\maketitle
\begin{abstract}
The sequentially Markov coalescent (SMC) is a Markov jump process
which models correlations in local genealogies across a chromosome.
It has been used as a theoretical tool for studying linkage disequilibrium
and identity-by-descent, and it also forms the basis of a class of
statistical procedures for estimating population history and inferring ancestry. 
In this paper, we study the rate at which SMC forgets its initial
condition in the pairwise setting. For the embedded jump chain,
we prove geometric ergodicity in total variation, with explicit constants.
For the continuous process, by contrast, the total variation distance from stationarity decays as $\asymp 1/\ell$ in genetic distance $\ell$. We obtain analogous
results for the closely related SMC' process using a novel time-change argument.
One application of these results is to justify heuristic approximations used in the literature that treat distant loci as evolving independently.
\end{abstract}
\global\long\def\dtv#1#2{\left\Vert #1-#2\right\Vert _{\text{TV}}}%

\global\long\def\E{\mathbb{E}}%

\section{Introduction}

The ancestral recombination graph (ARG) describes the joint genealogical
history of sampled genomes along a recombining chromosome. As originally conceptualized by Griffiths and Marjoram \parencite{griffiths1991two,griffiths1996ancestral,griffiths1997ancestral}, the ARG
is a partition-valued Markov process in which 
ancestral lineages undergo coalescence and recombination as one traces
their history into the past. A complementary spatial viewpoint was
later introduced by \textcite{wiuf1999recombination}, who showed that
the ARG can be equivalently constructed along the genome by updating the local genealogy
from left to right. This reformulation was important because it linked
coalescent theory to the sequential methods that were transforming
biological sequence analysis at the time \parencite{durbin1998biological,eddy1998profile}.

One difficulty in working with the Wiuf--Hein process is that it is not Markovian: the genealogy at a given position
depends on those at all previous positions. \textcite{mcvean2005approximating} introduced the sequentially
Markov coalescent (SMC) as a tractable Markov approximation to the
spatial ARG. Shortly thereafter, a refined version of this model, SMC' (``SMC-prime''),
was proposed to improve accuracy by allowing the process to have a limited form of memory \parencite{marjoram2006fast,wilton2015smc}.
The SMC and SMC' models now underlie many widely used methods for demographic inference,
recombination-rate estimation, and genealogical reconstruction \parencite{li2011inference,rasmussen2014genome,terhorst2017robust,deng2025robust}.

A basic question in stochastic process theory is how quickly the process forgets
its initial condition and converges to its stationary distribution. For the SMC-type processes, apart from theoretical interest, this question is practically relevant:
it determines the genomic scale
on which distant loci can be treated as approximately independent.
Thus it enters, for
example, into heuristic uses of effective independence such as LD pruning
in GWAS and, more recently, estimating population history using thinned marginal genealogies extracted from inferred ARGs \parencite{osmond2024dispersal,fan2025likelihood,dehaas2025inference,grundler2025geographic}.

In this paper we study the rates of forgetting for SMC and SMC' in the simplest case of
two chromosomes sampled from a panmictic population.
For the embedded jump chains, we prove geometric ergodicity in total
variation, with explicit bounds. In contrast, for the continuous processeses,
we show that the total variation
distance from stationarity decays only inversely with genetic distance. 
Thus the jump scale and the spatial scale have quantitatively different
forgetting behavior. 

The rest of the paper is organized as follows. In Section \ref{sec:preliminaries} we define the SMC and SMC' processes and record some of their basic properties. In Section \ref{sec:results} we state and prove our main results on the rates of forgetting for these processes. Finally, in Section \ref{sec:discussion} we discuss the implications of our results and directions for future work.

\section{Preliminaries}
\label{sec:preliminaries}

For $n=2$ samples, 
SMC describes the local time to most recent common ancestor (TMRCA) of a pair of recombining chromosomes along a recombining genome. 
Throughout the paper $\ell$ denotes genetic distance along the chromosome.

\begin{defn}[Sequentially Markov coalescent process]
\label{def:smc}Let $(Y_{\ell})_{\ell\ge0}$ be a Markov
jump process on $(0,\infty)$ with the following dynamics. When $Y_{\ell}=y$,
the process waits an exponential time $r(y)=y$ before jumping,
at which time the process takes on
the new value $Y'=Uy+Z$, where $U\sim\mathrm{Unif}(0,1)$ is the
recombination time on the current tree expressed as a fraction of
the tree height, and $Z\sim\mathrm{Exp}(1)$ is the additional coalescence
time after the recombination event. 

Equivalently, one may construct $(Y_{\ell})$ from its embedded jump
chain and holding times as follows. Let $(X_{n})_{n\ge0}$ be the
jump chain and let $(E_{n})_{n\ge0}$ be i.i.d.\ $\mathrm{Exp}(1)$,
independent of the chain. Given $X_{n}$, define the holding time
\[
\Delta\ell_{n}:=\frac{E_{n}}{X_{n}}\ \sim\ \mathrm{Exp}(X_{n}),
\]
and jump locations $\ell_{0}:=0$, $\ell_{n+1}:=\ell_{n}+\Delta\ell_{n}$.
Then define 
\begin{align}
Y_{\ell} & :=X_{n}\quad\text{for }\ell\in[\ell_{n},\ell_{n+1}),\nonumber \\
X_{n+1} & :=U_{n}X_{n}+Z_{n},\label{eq:Xn+1}
\end{align}
where $(U_{n},Z_{n})$ are i.i.d.\ with $U_{n}\sim\mathrm{Unif}(0,1)$
and $Z_{n}\sim\mathrm{Exp}(1)$. In particular, $(Y_{\ell})$ is almost surely
piecewise constant.
\end{defn}

The SMC' \parencite{marjoram2006fast} is a related
process with similar dynamics. Interest in the SMC' process arises
because it is more faithful to the underlying ancestral recombination
graph structure which both SMC and SMC' are trying to approximate
\parencite{wilton2015smc}.
\begin{defn}
\label{def:smc-prime} Let $(Y'_{\ell})_{\ell\ge0}$ be a Markov jump
process on $(0,\infty)$ with the following dynamics. When
$Y'_{\ell}=y$, recombination events occur at rate $y$. Given
that a recombination occurs at time
$r\in(0,y)$, the detached lineage recoalesces with either of the
two existing lineages at rate $1$ each (total rate $2$). If recoalescence
occurs with its own ancestral lineage before time $y$, then the recombination
is \emph{silent} and the state does not change. Otherwise---either
if recoalescence occurs with the other lineage, or if no recoalescence
occurs before time $y$---the recombination is \emph{visible}, and
the process jumps to a new TMRCA determined by the time of this coalescence
event.
The probability of a visible recombination at current TMRCA $s$ is
\begin{equation}
p_{\mathrm{vis}}(s)=\frac{2s+1-e^{-2s}}{4s},\qquad
p_{\mathrm{sil}}(s)=1-p_{\mathrm{vis}}(s).
\label{eq:p_vis}
\end{equation}

Equivalently, one may construct $(Y'_{\ell})$ from its embedded jump
chain and holding times as follows. Let $(X'_{n})_{n\ge0}$ be the
jump chain, and let $(E_{n})_{n\ge0}$ be i.i.d.\ $\mathrm{Exp}(1)$
random variables, independent of the chain. Given $X'_{n}$, define
the holding time 
\[
\Delta\ell_{n}:=\frac{E_{n}}{X'_{n} p_{\mathrm{vis}}(X_{n}')}
\]
and the jump locations $\ell_{0}:=0$, $\ell_{n+1}:=\ell_{n}+\Delta\ell_{n}$.
Then define 
\[
Y'_{\ell}:=X'_{n}\quad\text{for }\ell\in[\ell_{n},\ell_{n+1}),
\]
and $X'_{n+1}\sim q_{\mathrm{SMC}'}( \cdot \mid X_{n}'),$ where
\[
q_{\mathrm{SMC}'}(t\mid s)=\begin{cases}
\dfrac{2(1-e^{-2t})}{2s+1-e^{-2s}}, & 0<t<s,\\[10pt]
\dfrac{2(1-e^{-2s})}{2s+1-e^{-2s}} e^{-(t-s)}, & t\ge s.
\end{cases}
\]
\end{defn}

We record for use below the following standard facts about the SMC
and SMC' jump chains.

\begin{prop}
The SMC jump chain $(X_{n})$ defined by (\ref{eq:Xn+1}) has
transition density
\begin{equation}
q_{\mathrm{SMC}}(t\mid s)=\frac{e^{-t}\left(e^{s\wedge t}-1\right)}{s},\quad s>0, t>0.\label{eq:q(t|s)}
\end{equation}
It is irreducible and reversible with respect to 
\[
\mu(dt)=te^{-t} dt,
\]
and $\mu$ is the unique invariant probability measure.

The SMC' jump chain $(X'_{n})$ has transition density
$q_{\mathrm{SMC}'}(t\mid s)$ given above. 
It is irreducible and reversible with respect to 
\[
\mu'(dt)=\frac{3}{8} \bigl(2t+1-e^{-2t}\bigr)e^{-t} dt.
\]
and $\mu'$ is the unique invariant probability measure.
\end{prop}

Note that, whereas the SMC jump chain has a simple representation as a mixture of a uniform and an exponential distribution, the SMC' jump chain has a more complicated structure. This will become important in the analysis of the rates of forgetting for these two processes. 

\section{Results}
\label{sec:results}

In this section we prove our main results, which are sharp (up to
constants) estimates of the rates of forgetting for the SMC and SMC',
for the continuous processes, as well as upper bounds on the forgetting rates of the embedded jump chains.

For these results we work with total variation distance: if $\mu$
and $\pi$ are two measures defined on a common space, then the total
variation distance is defined to be 
\[
d_{\text{TV}}(\mu,\pi)=\sup_{A}\left|\mu(A)-\pi(A)\right|
\]
where the supremum is over all measurable subsets $A$. If $\mu$
and $\pi$ possess densities $f_{\mu}$ and $f_{\pi}$ then equivalently
\[
d_{\text{TV}}(\mu,\pi)=\frac{1}{2}\int\left|f_{\mu}-f_{\pi}\right|.
\]
When $X\sim\mu$ and $Y\sim\pi$, we sometimes abuse notation and
write $d_{\text{TV}}(X,Y)$ to mean either of the above definitions
if there is no risk of confusion.

\subsection{The jump chains}

The goal of this section is to prove exponential forgetting for the
embedded jump chains of the SMC and SMC' processes. This is accomplished
using the following result, which bounds the rate of forgetting
for chains possessing certain contractivity properties. This recalls the well known Foster-Lyapunov method \parencite{roberts2004general,meyn2012markov}, and we could also have employed that machinery here.
An advantage of \Propref{tv-general} is that it is simpler and yields better constants.
\begin{prop}
\label{prop:tv-general} Let $K$ be a Markov kernel on $(0,\infty)$
with transition density $q( \cdot |s)$ with respect to Lebesgue
measure, and let $\pi$ be an invariant probability measure for $K$.
Suppose that:
\begin{enumerate}
\item There exists $L<\infty$ such that 
\begin{equation}
\int^{\infty}_{0}\bigl|q(t|a)-q(t|b)\bigr| dt\le L |a-b|,\qquad a,b>0;\label{eq:l1-lipschitz}
\end{equation}
\item There exists a coupling $(X_{n},Y_{n})_{n\ge0}$ of two copies of
the chain, with 
\[
X_{0}=x,\qquad Y_{0}\sim\pi,
\]
such that for some $r\in(0,1)$, 
\begin{equation}
\mathbb{E}\!\left[ |X_{n+1}-Y_{n+1}| \, \mid  \, X_{n},Y_{n}\right]\le r |X_{n}-Y_{n}|\qquad\text{a.s. for all }n\ge0.\label{eq:contractive-coupling}
\end{equation}
Then for every $n\ge1$, 
\begin{equation}
d_{\mathrm{TV}}\!\bigl(K^{n}(x,\cdot),\pi\bigr)\le\frac{L}{2} r^{ n-1} \mathbb{E}|x-Y_{0}|.\label{eq:general-tv-bound}
\end{equation}
In particular, the chain is geometrically ergodic in total variation.
\end{enumerate}
\end{prop}
\begin{proof}
Let $f_{n}$ denote the density of $K^{n}(x,\cdot)$, and also write
$\pi(t)$ for the stationary density of $\pi$. By the Markov property
and stationarity, 
\[
f_{n}(t)=\mathbb{E}\!\left[q(t|X_{n-1})\right],\qquad\pi(t)=\mathbb{E}\!\left[q(t|Y_{n-1})\right].
\]
Therefore 
\begin{align*}
2 d_{\mathrm{TV}}\!\bigl(K^{n}(x,\cdot),\pi\bigr) & =\int^{\infty}_{0}|f_{n}(t)-\pi(t)| dt\\
 & \le\int^{\infty}_{0}\mathbb{E}\!\left[\bigl|q(t|X_{n-1})-q(t|Y_{n-1})\bigr|\right]dt\\
 & =\mathbb{E}\!\left[\int^{\infty}_{0}\bigl|q(t|X_{n-1})-q(t|Y_{n-1})\bigr|dt\right]\\
 & \le L \mathbb{E}|X_{n-1}-Y_{n-1}|
\end{align*}
by (\ref{eq:l1-lipschitz}). Iterating (\ref{eq:contractive-coupling})
yields 
\[
\mathbb{E}|X_{n-1}-Y_{n-1}|\le r^{ n-1} \mathbb{E}|x-Y_{0}|.
\]
Substituting this into the previous display gives (\ref{eq:general-tv-bound}). 
\end{proof}

Next we develop some lemmas which are useful for verifying the conditions of the proposition.
In the following result, $L^1(0,\infty)$ denotes the Banach space of integrable functions on $(0,\infty)$ with norm $\|f\|_{L^1}=\int_0^\infty |f(t)| dt$.
\begin{lem}
\label{lem:l1bound}Let $\{q_{x}(\cdot):x>0\}$ be a family of probability
densities on $(0,\infty)$, and regard $x\mapsto q_{x}(\cdot)$ as
a map from $(0,\infty)$ into $L^{1}(0,\infty)$.
Assume there exist continuously differentiable functions $A,B:(0,\infty)\to\mathbb{R}$
and an integrable function $\Phi:(0,\infty)\to\mathbb{R}$ such that
the following all hold:
\begin{gather}
q_{x}(t)=A(x) \Phi(t) \mathbf{1}_{\{0<t<x\}}+B(x) e^{-t} \mathbf{1}_{\{t\ge x\}},\qquad t>0.\label{eq:kernel-form}\\
\lim_{t\uparrow x}q_{x}(t)=A(x)\Phi(x)=B(x)e^{-x}=\lim_{t\downarrow x}q_{x}(t).\label{eq:matching}\\
\partial_{x}q_{x}(t):=A'(x) \Phi(t) \mathbf{1}_{\{0<t<x\}}+B'(x) e^{-t} \mathbf{1}_{\{t>x\}}\in L^{1}(0,\infty)\label{eq:formal-derivative}\\
\sup_{x>0}\|\partial_{x}q_{x}\|_{L^{1}}\le L<\infty.\label{eq:deriv-bound}
\end{gather}
Then $x\mapsto q_{x}$ is absolutely continuous as an $L^{1}$-valued
function, and for all $a,b>0$, 
\begin{equation}
q_{b}-q_{a}=\int^{b}_{a}\partial_{x}q_{x} dx,\quad\text{in \ensuremath{L^{1}.}}\label{eq:fundamental-theorem}
\end{equation}
Furthermore, for all $a,b>0$, 
\begin{equation}
\|q_{b}-q_{a}\|_{L^{1}}\le L |b-a|.\label{eq:l1-lipschitz-conclusion}
\end{equation}
\end{lem}
\begin{proof}
Fix $0<a<b$. For each fixed $t>0$, the function $x\mapsto q_{x}(t)$
has the form
\[
q_{x}(t)=
\begin{cases}
B(x)e^{-t}, & 0<x\le t,\\
A(x)\Phi(t), & x>t.
\end{cases}
\]
Thus $x\mapsto q_{x}(t)$ is continuously differentiable away from
$x=t$. At $x=t$, the one-sided values agree by (\ref{eq:matching}):
\[
\lim_{x\uparrow t}q_{x}(t)=B(t)e^{-t}=A(t)\Phi(t)=\lim_{x\downarrow t}q_{x}(t).
\]
Therefore $x\mapsto q_{x}(t)$ is absolutely continuous on $[a,b]$,
and its derivative is $\partial_{x}q_{x}(t)$ for almost every $x\in(a,b)$.
This implies that for every $t>0$,
\begin{equation}
q_{b}(t)-q_{a}(t)=\int^{b}_{a}\partial_{x}q_{x}(t) dx.\label{eq:l1_ftc}
\end{equation}
Moreover, by (\ref{eq:deriv-bound}),
\[
\int^{\infty}_{0}\int^{b}_{a}|\partial_{x}q_{x}(t)| dx dt\le\int^{b}_{a}\|\partial_{x}q_{x}\|_{L^{1}} dx\le L(b-a)<\infty.
\]
Hence Fubini's theorem shows that $t\mapsto\int^{b}_{a}\partial_{x}q_{x}(t) dx$
belongs to $L^{1}(0,\infty)$. Since $q_{b}-q_{a}\in L^{1}(0,\infty)$
as well, (\ref{eq:l1_ftc}) yields (\ref{eq:fundamental-theorem}).
Taking $L^{1}$ norms and applying Minkowski's integral inequality
gives
\[
\|q_{b}-q_{a}\|_{L^{1}}=\left\Vert \int^{b}_{a}\partial_{x}q_{x} dx\right\Vert _{L^{1}}\le\int^{b}_{a}\|\partial_{x}q_{x}\|_{L^{1}} dx\le\int^{b}_{a}L dx=L|b-a|,
\]
which proves (\ref{eq:l1-lipschitz-conclusion}).
\end{proof}

\begin{lem}
\label{lem:l1-contraction}Let $q(t\mid x)$ be a probability distribution
in $t$ and assume that it is stochastically monotone, i.e.
\[
F_{x}(y):=\int^{y}_{0}q(t|x) dt
\]
is non-increasing in $x$ for fixed $y>0$. Further define the mean
map 
\[
m(x):=\int^{\infty}_{0}t q(t|x) dt.
\]
If there exists $0<r<1$ such that 
\begin{equation}
|m(a)-m(b)|\le r |a-b|,\qquad a,b>0,\label{eq:mean-contraction}
\end{equation}
then the quantile coupling 
\[
X_{n+1}=F^{-1}_{X_{n}}(U_{n}),\qquad Y_{n+1}=F^{-1}_{Y_{n}}(U_{n}),\qquad U_{n}\sim\mathrm{Unif}(0,1),
\]
satisfies (\ref{eq:contractive-coupling}). 
\end{lem}
\begin{proof}
Stochastic monotonicity implies that if $X_{n}\ge Y_{n}$ then $X_{n+1}\ge Y_{n+1}$
a.s., and hence 
\[
\mathbb{E}\!\left[|X_{n+1}-Y_{n+1}| \, \mid \, X_{n},Y_{n}\right]=m(X_{n})-m(Y_{n})\le r |X_{n}-Y_{n}|.
\]
The case $X_{n}\le Y_{n}$ is identical.
\end{proof}

Using the preceding, we can now prove the main results of this section
on the rates of forgetting for the SMC and SMC' jump chains.
\begin{thm}[Geometric ergodicity of the SMC jump chain]
\label{cor:smc-jump-chain} Let $K_{\mathrm{SMC}}$ be the jump kernel
with density $q_{\mathrm{SMC}}(t\mid s)$ and invariant law $\mu$.
Then
\[
d_{\mathrm{TV}}\!\bigl(K^{n}_{\mathrm{SMC}}(x,\cdot),\mu\bigr)\le\frac{(5/2)(x+2).}{2^{n}}
\]
\end{thm}
\begin{proof}
We verify the conditions of \Propref{tv-general}. First, we check
(\ref{eq:l1-lipschitz}) with $L=5/2$ using \Lemref{l1bound}. Write
\[
q_{\mathrm{SMC}}(t\mid x)=q_{x}(t):=A(x) \Phi(t) \mathbf{1}_{\{0<t<x\}}+B(x)e^{-t} \mathbf{1}_{\{t\ge x\}},
\]
where 
\[
A(x):=\frac{1}{x},\qquad B(x):=\frac{e^{x}-1}{x},\qquad\Phi(t):=1-e^{-t}.
\]
The matching condition holds because 
\[
A(x)\Phi(x)=\frac{1-e^{-x}}{x}=\frac{e^{x}-1}{x}e^{-x}=B(x)e^{-x}.
\]
Moreover, 
\[
\partial_{x}q_{x}(t)=-\frac{1-e^{-t}}{x^{2}} \mathbf{1}_{\{0<t<x\}}+B'(x)e^{-t} \mathbf{1}_{\{t>x\}},
\]
where 
\[
B'(x)=\frac{xe^{x}-e^{x}+1}{x^{2}}.
\]
Since 
\[
e^{-x}B'(x)=\frac{x-1+e^{-x}}{x^{2}}\ge0,
\]
the derivative is nonpositive on $(0,x)$ and nonnegative on $(x,\infty)$.
Because $\int^{\infty}_{0}q_{x}(t) dt=1$, differentiation gives
\[
\int^{\infty}_{0}\partial_{x}q_{x}(t) dt=0,
\]
and therefore 
\[
\|\partial_{x}q_{x}\|_{L^{1}}=2\int^{\infty}_{x}\partial_{x}q_{x}(t) dt=2e^{-x}B'(x)=2 \frac{x-1+e^{-x}}{x^{2}}.
\]
We claim that 
\[
2 \frac{x-1+e^{-x}}{x^{2}}\le\frac{5}{2}\qquad(x>0)
\]
which is equivalent to 
\[
f(x):=\frac{5}{2}x^{2}-2x+2-2e^{-x}\ge0.
\]
Indeed, $f(0)=0$ and 
\[
f'(x)=5x-2+2e^{-x},\qquad f''(x)=5-2e^{-x}\ge3>0.
\]
Thus $f'$ is increasing with $f'(0)=0$, so $f'(x)\ge0$ for all
$x\ge0$, hence $f(x)\ge0$ for all $x\ge0$. Therefore 
\[
\sup_{x>0}\|\partial_{x}q_{x}\|_{L^{1}}\le\frac{5}{2}.
\]
So condition (\ref{eq:l1-lipschitz}) holds with $L=5/2$.

Next, we verify condition (\ref{eq:contractive-coupling}). Using
the representation $X_{n+1}=U_{n}X_{n}+Z_{n}$ with $U_{n}\sim\mathrm{Unif}(0,1)$
and $Z_{n}\sim\mathrm{Exp}(1)$, we immediately have
\[
\mathbb{E}\left[|X_{n+1}-Y_{n+1}| \mid X_{n},Y_{n}\right]\le\frac{1}{2}|X_{n}-Y_{n}|.
\]
So (\ref{eq:contractive-coupling}) holds with $r=1/2$.

Applying \Propref{tv-general} therefore gives 
\[
d_{\mathrm{TV}}\!\bigl(K^{n}_{\mathrm{SMC}}(x,\cdot),\mu\bigr)\le\frac{5/2}{2}\left(\frac{1}{2}\right)^{n-1}\mathbb{E}|x-Y_{0}|,\qquad Y_{0}\sim\mu.
\]
Since $\mu$ has mean $2$, 
\[
\mathbb{E}|x-Y_{0}|\le x+\mathbb{E}[Y_{0}]=x+2.
\]
Therefore 
\[
d_{\mathrm{TV}}\!\bigl(K^{n}_{\mathrm{SMC}}(x,\cdot),\mu\bigr)\le\frac{5(x+2)}{2^{ n+1}},
\]
as claimed.
\end{proof}

Next we prove the analogous result for the jump chain of the SMC'
process.
\begin{thm}[Geometric ergodicity of the SMC' jump chain]
\label{thm:smcprime-geo} Let $K_{\mathrm{SMC}'}$ be the Markov
jump kernel with density $q_{\mathrm{SMC}'}(t\mid s)$ and invariant
law $\mu'$.
Then for every $n\ge1$,
\[
d_{\mathrm{TV}}\!\bigl(K^{n}_{\mathrm{SMC}'}(x,\cdot),\mu'\bigr)\le\frac{x+\frac{11}{6}}{2^{ n-1}}.
\]
In particular, the SMC' jump chain is geometrically ergodic in total
variation.
\end{thm}
\begin{proof}
We verify the conditions \Propref{tv-general}. First, we check (\ref{eq:l1-lipschitz})
with $L=2$ using \Lemref{l1bound}. Write 
\[
q_{\mathrm{SMC}'}(t\mid x)=q_{x}(t):=A(x) \Phi(t) \mathbf{1}_{\{0<t<x\}}+B(x)e^{-t} \mathbf{1}_{\{t\ge x\}},
\]
where 
\[
\begin{aligned}
A(x) &:= \frac{2}{D(x)}, \qquad & B(x) &:= \frac{2(e^{x}-e^{-x})}{D(x)},\\
\Phi(t) &:= 1-e^{-2t}, \qquad & D(x) &:= 2x+1-e^{-2x}.
\end{aligned}
\]
The matching condition holds because 
\[
\begin{aligned}
A(x)\Phi(x)
&= \frac{2(1-e^{-2x})}{D(x)}\\
&= \frac{2(e^{x}-e^{-x})}{D(x)}e^{-x}\\
&= B(x)e^{-x}.
\end{aligned}
\]
Moreover, 
\[
\partial_{x}q_{x}(t)=-\frac{4(1+e^{-2x})(1-e^{-2t})}{D(x)^{2}} \mathbf{1}_{\{0<t<x\}}+B'(x)e^{-t} \mathbf{1}_{\{t>x\}}.
\]
A calculation shows that
\[
e^{-x}B'(x)=\frac{2(1+e^{-2x})(2x-1+e^{-2x})}{D(x)^{2}}\ge0,
\]
since $2x-1+e^{-2x}\ge2x-1+(1-2x)=0$. Thus $\partial_{x}q_{x}(t)\le0$
on $(0,x)$ and $\partial_{x}q_{x}(t)\ge0$ on $(x,\infty)$. Because
$\int^{\infty}_{0}q_{x}(t) dt=1$, differentiation yields 
\[
\int^{\infty}_{0}\partial_{x}q_{x}(t) dt=0,
\]
and therefore 
\[
\begin{aligned}
\|\partial_{x}q_{x}\|_{L^{1}}
&= 2\int^{\infty}_{x}\partial_{x}q_{x}(t) dt\\
&= 2e^{-x}B'(x)\\
&= \frac{4(1+z)(2x-1+z)}{(2x+1-z)^{2}},
\qquad z:=e^{-2x}.
\end{aligned}
\]
We claim that 
\[
\frac{4(1+z)(2x-1+z)}{(2x+1-z)^{2}}\le2.
\]
This is equivalent to 
\[
(2x+1-z)^{2}-2(1+z)(2x-1+z)\ge0.
\]
Expanding gives 
\[
(2x+1-z)^{2}-2(1+z)(2x-1+z)=4x^{2}-8xz-z^{2}-2z+3=:\Phi_{x}(z).
\]
For fixed $x>0$, $\Phi_{x}$ is strictly decreasing in $z$: 
\[
\partial_{z}\Phi_{x}(z)=-8x-2z-2<0.
\]
Since $z=e^{-2x}\le(1+2x)^{-1}$, we have 
\[
\Phi_{x}(z)\ge\Phi_{x}\!\left(\frac{1}{1+2x}\right)=\frac{16x^{3}(x+1)}{(2x+1)^{2}}\ge0.
\]
Hence 
\[
\sup_{x>0}\|\partial_{x}q_{x}\|_{L^{1}}\le2.
\]
Lemma~\ref{lem:l1bound} now yields 
\[
\int^{\infty}_{0}\left|q_{a}(t)-q_{b}(t)\right| dt\le2|a-b|,\qquad a,b>0.
\]
So condition~(\ref{eq:l1-lipschitz}) holds with $L=2$.

Next, we verify condition~(\ref{eq:contractive-coupling}) via \Lemref{l1-contraction}. We have 
\[
F_{x}(y):=\int^{y}_{0}q_{x}(t) dt=\begin{cases}
\dfrac{2y-1+e^{-2y}}{D(x)}, & 0<y<x,\\[10pt]
1-B(x)e^{-y}, & y\ge x.
\end{cases}
\]
For fixed $y>0$, the map $x\mapsto F_{x}(y)$ is non-increasing:
if $x>y$, this is immediate from the first formula since $D$ is
increasing; if $x\le y$, this follows because 
$B'(x)\ge0$,
so $B$ is increasing and hence $1-B(x)e^{-y}$ is decreasing in $x$.
Thus the kernel is stochastically monotone. Its mean map is 
\[
m(x):=\int^{\infty}_{0}t q_{x}(t) dt=\frac{x^{2}+2x+\frac{3}{2}-(x+\frac{3}{2})e^{-2x}}{2x+1-e^{-2x}}.
\]
Differentiating yields
\[
\frac{1}{2}-m'(x)=\frac{R(x)}{2\left[(2x+1)e^{2x}-1\right]^{2}},
\]
where 
\[
R(x):=3e^{4x}-(4x^{2}+8x+2)e^{2x}-1.
\]
Now $R(0)=0$, and 
\[
R'(x)=4e^{2x}G(x),\qquad G(x):=3e^{2x}-3-6x-2x^{2}.
\]
Also $G(0)=G'(0)=0$, while 
\[
G''(x)=12e^{2x}-4>0.
\]
Hence $G'(x)\ge0$, so $G(x)\ge0$, and therefore $R'(x)\ge0$. Since
$R(0)=0$, it follows that $R(x)\ge0$ for all $x\ge0$. Thus 
\[
m'(x)\le\frac{1}{2}\qquad(x>0),
\]
and therefore 
\[
|m(a)-m(b)|\le\frac{1}{2}|a-b|,\qquad a,b>0.
\]
By \Lemref{l1-contraction}, the quantile coupling satisfies 
\[
\mathbb{E}\!\left[|X_{n+1}-Y_{n+1}| \mid X_{n},Y_{n}\right]\le\frac{1}{2}|X_{n}-Y_{n}|.
\]
So condition (\ref{eq:contractive-coupling}) holds with $r=1/2$.
Applying \Propref{tv-general} therefore gives 
\[
d_{\mathrm{TV}}\!\bigl(K^{n}_{\mathrm{SMC}'}(x,\cdot),\mu'\bigr)\le\frac{2}{2}\left(\frac{1}{2}\right)^{n-1}\mathbb{E}|x-Y_{0}|,\qquad Y_{0}\sim\mu'.
\]
Finally, 
\[
\mathbb{E}|x-Y_{0}|\le x+\mathbb{E}_{\mu'}[Y_{0}].
\]
Since 
\[
\mathbb{E}_{\mu'}[Y_{0}]=\frac{3}{8}\int^{\infty}_{0}y(2y+1-e^{-2y})e^{-y} dy=\frac{11}{6},
\]
we obtain 
\[
d_{\mathrm{TV}}\!\bigl(K^{n}_{\mathrm{SMC}'}(x,\cdot),\mu'\bigr)\le\frac{x+\frac{11}{6}}{2^{ n-1}},
\]
as claimed.
\end{proof}

\subsection{The continuous processes}

Next we turn our attention to the continuous processes into which the
above jump chains are embedded. For both continuous processes,
we write $\pi(dy)=e^{-y} dy$ for the stationary law.

\subsubsection{SMC process}

Recall from Definition \ref{def:smc} the definition of the SMC process
$(Y_{\ell})_{\ell\ge0}.$ For bounded measurable $f:(0,\infty)\to\mathbb{R}$
define the jump operator 
\begin{equation}
Jf(y):=\mathbb{E}\bigl[f(Uy+Z)\bigr].\label{eq:Jf(y)}
\end{equation}
It follows from the definition that the generator $A$ of $(Y_{\ell})$
is 
\begin{equation}
Af(x)=x\bigl(Jf(x)-f(x)\bigr),\label{eq:generator}
\end{equation}
for all $f$ in the generator domain.

Using a short probabilistic argument, \textcite{durrett2008probability} derived the following expression for the transition kernel of the SMC process. As a mathematical curiosity, in the appendix we give an alternative proof of this result by deriving the Laplace transform of the transition kernel. This analytic approach might prove useful for analyzing other genealogical processes where the probabilistic approach fails.

\begin{prop}
\label{thm:smc-transition}Let $P_{\ell}(x,dy)$ denote the transition
kernel of the SMC process. Then 
\begin{equation}
P_{\ell}(x,dy)=e^{-\ell x} \delta_{x}(dy)+p_{\ell}(x,y) dy,\label{eq:P_ell}
\end{equation}
where the density $p_{\ell}(x,y)$ is 
\begin{equation}
p_{\ell}(x,y)=\begin{cases}
\frac{\ell}{\ell-1}e^{-y}\left(1-e^{-(\ell-1)\min(x,y)}\right), & \ell\neq1\\
e^{-y}\min(x,y), & \ell=1.
\end{cases}\label{eq:p_c}
\end{equation}
\end{prop}

Using the joint density representation, we may show that, in contrast
to its embedded jump chain, the SMC process is only polynomially ergodic.
\begin{thm}
\label{thm:smc-tv-sharp} Fix $x>0$. Let $(Y_{\ell})_{\ell\ge0}$
be the SMC process with transition kernel $P_{\ell}(x,\cdot)$ and
stationary law $\pi$. Then there exists $\ell_{0}(x)$
such that for all $\ell\ge \ell_{0}(x)$, 
\[
\frac{1}{2\ell}\le d_{\mathrm{TV}}\!\bigl(P_{\ell}(x,\cdot),\pi\bigr)\le\frac{1}{\ell}.
\]
In particular, 
\[
d_{\mathrm{TV}}\!\bigl(P_{\ell}(x,\cdot),\pi\bigr)\asymp\frac{1}{\ell}\qquad\text{as }\ell\to\infty.
\]
\end{thm}
\begin{proof}
Fix $x>0$ and assume $\ell>1$. We have
\[
P_{\ell}(x,\cdot)-\pi=e^{-\ell x}\delta_{x}+\bigl(p_{\ell}(x,y)-e^{-y}\bigr) dy,
\]
and 
\[
e^{-y}-p_{\ell}(x,y)=\frac{e^{-y}}{\ell-1}\Bigl(\ell e^{-(\ell-1)\min(x,y)}-1\Bigr).
\]
Let 
\[
y_{*}:=\frac{\log \ell}{\ell-1},
\]
which is smaller than $x$ for all $\ell>\ell_0(x)$.
Therefore:
\begin{itemize}
\item For $0<y<y_{*}<x$, we have $\min(x,y)=y$ and 
\[
\ell e^{-(\ell-1)\min(x,y)}
> \ell e^{-(\ell-1)y_*}
=1,
\]
so $e^{-y}-p_{\ell}(x,y)>0$;
\item For $y>y_{*}$, we have $\min(x,y)>y_{*}$ and therefore
\[
\ell e^{-(\ell-1)\min(x,y)}<
\ell e^{-(\ell-1)y_*}
=1,
\]
so $e^{-y}-p_{\ell}(x,y)<0$. 
\end{itemize}
Thus the absolutely continuous part of $P_{\ell}(x,\cdot)-\pi$ is negative
on $(0,y_{*})$ and positive on $(y_{*},\infty)$, while the atom
$e^{-\ell x}\delta_{x}$ is positive.
Since $P_{\ell}(x,\cdot)-\pi$ has total mass zero, the total variation
norm equals the mass of either the positive or negative part. Choosing the negative part, we get
\begin{align*}
d_{\mathrm{TV}}\!\bigl(P_{\ell}(x,\cdot),\pi\bigr)&=\int^{y_{*}}_{0}\bigl(e^{-y}-p_{\ell}(x,y)\bigr) dy \\
 & =\frac{1}{\ell-1}\int^{y_{*}}_{0}\Bigl(\ell e^{-\ell y}-e^{-y}\Bigr) dy\\
 & =\frac{e^{-y_{*}}-e^{-\ell y_{*}}}{\ell-1}\\
 & =\frac{\ell^{-1/(\ell-1)}\left(1-\ell^{-1}\right)}{\ell-1}\\
 & =\frac{\ell^{-1/(\ell-1)}}{\ell}.
\end{align*}
Since $h(\ell):=\frac{\log \ell}{\ell-1}$
is decreasing for $\ell \ge 2$,
\[
\ell^{-1/(\ell-1)}=e^{-h(\ell)}\in\left[\frac{1}{2},1\right].
\]
Therefore, after increasing $\ell_{0}(x)$ if necessary so that $\ell_{0}(x)\ge2$,
we obtain 
\[
\frac{1}{2\ell}\le d_{\mathrm{TV}}\!\bigl(P_{\ell}(x,\cdot),\pi\bigr)\le\frac{1}{\ell},\qquad \ell\ge \ell_{0}(x),
\]
as claimed.
\end{proof}

\subsubsection{SMC' process}

Next we turn our attention to the SMC' process (\Defref{smc-prime}).
Recall that under this model, recombination still occurs at rate $x$,
but some fraction of recombinations are ``silent'' and do not result
in TMRCA changing; see (\ref{eq:p_vis}).
The effective rate of ``visible'' recombination (i.e.,~one which
causes TMRCA to change from $s$ to $t\neq s$) is $s p_{\mathrm{vis}}(s)$.
Conditional on a visible recombination, the post-jump density is $q_{\mathrm{SMC}'}(t\mid s)$.
Denote by $A'$ the generator of the SMC' process:
\begin{equation}
A'f(x)=\lim_{\ell\to0}\frac{\mathbb{E}f(Y'_{\ell})-f(x)}{\ell}=x\int^{\infty}_{0}\left[f(y)-f(x)\right]m_{x}(y) dy\label{eq:A' generator}
\end{equation}
where $m_{x}(y)=p_{\mathrm{vis}}(x)q_{\mathrm{SMC}'}(y\mid x)$. 

Extending the proof techniques of the previous section to the SMC'
model is complicated by the fact that innovations under this model
no longer have the nice algebraic form (\ref{eq:Jf(y)}). Instead
we proceed differently, by showing how to obtain the SMC' process
by a random time change of the SMC process.
\begin{thm}
Let $(Y_{\ell})_{\ell\ge0}$ denote the SMC process from Definition~\ref{def:smc}.
Let $N_{\ell}$ be a Poisson process of rate $1/4$ independent of $(Y_{\ell})$.
Define
the increasing process,
\begin{equation}
C_{\ell}:=\frac{\ell}{2}+2N_{\ell}\label{eq:C_ell}
\end{equation}
and the subordinated process $Y'_{\ell}=Y_{C_{\ell}}$. Then $Y'_{\ell}$
is equal in law to the SMC' process defined above. 
\end{thm}
\begin{proof}
Let $(P_{\ell}f)(x):=\mathbb{E}_{x}[f(Y_{\ell})]$ be the semigroup of the
ordinary SMC process, where $A$ is the generator (\ref{eq:generator}).
Define the subordinated semigroup $(P'_{\ell}f)(x):=\mathbb{E}\left[(P_{C_{\ell}}f)(x)\right]$
and let $a=\ell/4$. Since $N_{\ell}\sim\mathrm{Poisson}(a)$,
\[
P'_{\ell}f=(1-a)P_{2a}f+aP_{2(a+1)}f+o(\ell),\quad\text{as \ensuremath{\ell\to0}.}
\]
For the first term, 
\[
(1-a)P_{2a}f=(1-a)[f+2aA f+o(\ell)]=(1-a)f+2aA f+o(\ell).
\]
For the second term,
\[
aP_{2(a+1)}f=aP_{2}f+o(\ell).
\]
Hence,
\[
P_{\ell}'=f+a(P_{2}-I)f+2aA f+o(\ell)
\]
so that the generator of the subordinated process is
\[
\tilde{A}:=\lim_{\ell\to0}\frac{P'_{\ell}f-f}{\ell}=\frac{1}{2}\left[A f+\frac{1}{2}\left(P_{2}-I\right)f\right].
\]

We now compute this operator explicitly and show that it agrees with
$A'$ defined in (\ref{eq:A' generator}). First, by definition of
$A$, 
\[
\frac{1}{2}A f(x)=\frac{1}{2}\int^{\infty}_{0}[f(y)-f(x)] xq_{\mathrm{SMC}}(y\mid x) dy.
\]
By \Thmref{smc-transition}, 
\[
P_{2}(x,dy)=e^{-2x}\delta_{x}(dy)+p_{2}(x,y) dy,
\]
where 
\[
p_{2}(x,y)=\begin{cases}
2(e^{-y}-e^{-2y}), & 0<y<x,\\[6pt]
2(1-e^{-x})e^{-y}, & y\ge x.
\end{cases}
\]
Thus 
\[
\frac{1}{4}(P_{2}-I)f(x)=\frac{1}{4}(e^{-2x}-1)f(x)+\frac{1}{4}\int^{\infty}_{0}p_{2}(x,y)f(y) dy.
\]
Now we identify the off-diagonal jump density of $\tilde{A}$.
For $0<y<x$, using 
\[
xq_{\mathrm{SMC}}(y\mid x)=1-e^{-y},\qquad p_{2}(x,y)=2(e^{-y}-e^{-2y}),
\]
the coefficient of $f(y)$ in $\tilde{A}f(x)$ is 
\[
\frac{1}{2}(1-e^{-y})+\frac{1}{4}\cdot2(e^{-y}-e^{-2y})=\frac{1}{2}(1-e^{-2y}).
\]
Similarly, for $y\ge x$, using 
\[
xq_{\mathrm{SMC}}(y\mid x)=(e^{x}-1)e^{-y},\qquad p_{2}(x,y)=2(1-e^{-x})e^{-y},
\]
the coefficient of $f(y)$ is 
\[
\frac{1}{2}(e^{x}-1)e^{-y}+\frac{1}{4}\cdot2(1-e^{-x})e^{-y}=\frac{1}{2}(e^{x}-e^{-x})e^{-y}=\frac{1}{2}(1-e^{-2x})e^{-(y-x)}.
\]
Thus the off-diagonal jump measure of $\tilde{A}$ is $x m_{x}(y) dy,$where
\[
m_{x}(y)=\begin{cases}
\dfrac{1-e^{-2y}}{2x}, & 0<y<x,\\[10pt]
\dfrac{1-e^{-2x}}{2x}e^{-(y-x)}, & y\ge x,
\end{cases}
\]
exactly as in the definition of the SMC' generator.

It remains to check the diagonal coefficient. From the two terms above,
the coefficient of $-f(x)$ is 
\[
\frac{1}{2}x+\frac{1}{4}(1-e^{-2x})=x\left(\frac{1}{2}+\frac{1-e^{-2x}}{4x}\right)=x p_{\mathrm{vis}}(x),
\]
where $p_{\mathrm{vis}}(x)$ was defined in (\ref{eq:p_vis}). Adding
together these terms and simplifying, we obtain 
\[
\tilde{A}f(x)=x\int^{\infty}_{0}[f(y)-f(x)] m_{x}(y) dy=A'f(x).
\]
\end{proof}

The preceding result allows us to use what we already know about ergodicity
of the SMC process to establish the same type of result for the SMC'
process. 

\begin{thm}
\label{thm:smcprime-tv-theta} Let $(Y'_{\ell})_{\ell\ge0}$ be the
SMC' process, and let $P'_{\ell}(x,\cdot)$ denote its
transition kernel and $\pi$ its stationary law. Then
for every fixed $x>0$, 
\[
d_{\mathrm{TV}}\!\bigl(P'_{\ell}(x,\cdot),\pi\bigr)\asymp\frac{1}{\ell}\qquad\text{as }\ell\to\infty.
\]
\end{thm}
\begin{proof}
By the preceding theorem, $P'_{\ell}(x,\cdot)=\mathbb{E}\!\left[P_{C_{\ell}}(x,\cdot)\right]$
where $P_{\ell}(x,\cdot)$ is the transition kernel of the ordinary
SMC process. By convexity of total variation distance, 
\[
d_{\mathrm{TV}}\!\bigl(P'_{\ell}(x,\cdot),\pi\bigr)=d_{\mathrm{TV}}\!\bigl(\mathbb{E}P_{C_{\ell}}(x,\cdot),\pi\bigr)\le\mathbb{E}\!\left[d_{\mathrm{TV}}\!\bigl(P_{C_{\ell}}(x,\cdot),\pi\bigr)\right].
\]
Since $C_{\ell}\ge \ell/2$ almost surely, by Theorem \ref{thm:smc-tv-sharp} we have for sufficiently large 
$\ell$
\[
\mathbb{E}\left[d_{\mathrm{TV}}\!\bigl(P_{C_{\ell}}(x,\cdot),\pi\bigr)\right]\le\frac{2}{\ell}.
\]

We now prove a matching lower bound. Similar to the proof of \Thmref{smc-tv-sharp}, the strategy is to isolate the negative part of the density of $P_{C_{\ell}}(x,\cdot)-\pi$ and show that it has mass at least $\Omega(1/\ell)$.
Let 
\[
I_{\ell}:=\left[\frac{\ell}{2},\frac{3\ell}{2}\right],\qquad A_{\ell}:=\left(0,\frac{2}{3\ell}\right).
\]
We will show that for sufficiently large $\ell$, $P_{r}(x,\cdot)-\pi$ is negative on $A_{\ell}$ for all $r\in I_{\ell}$, and that the mass of this negative part is at least $\Omega(1/\ell)$ uniformly in $r\in I_{\ell}$. Then we will show that $\mathbb{P}(C_{\ell}\in I_{\ell})$ is bounded away from zero, which will yield the desired lower bound.

For the ordinary SMC kernel, when $r\neq1$ we have
\[
P_{r}(x,dy)=e^{-rx}\delta_{x}(dy)+p_{r}(x,y) dy,
\]
with 
\[
p_{r}(x,y)=\frac{r}{r-1}e^{-y}\Bigl(1-e^{-(r-1)\min(x,y)}\Bigr),
\]
and therefore 
\[
e^{-y}-p_{r}(x,y)=\frac{e^{-y}}{r-1}\Bigl(r e^{-(r-1)\min(x,y)}-1\Bigr).
\]

Since $x>0$ is fixed, there exists $\ell_{1}(x)$ such that
$A_{\ell}\subset(0,x)$ for all $\ell\ge\ell_{1}(x)$. Enlarging
$\ell_{1}(x)$ if necessary, we may also assume $\ell_{1}(x)\ge 4e$.
Fix $\ell\ge\ell_{1}(x)$, $r\in I_{\ell}$, and $y\in A_{\ell}$.
Then $y<x$, so
\[
e^{-y}-p_{r}(x,y)=\frac{e^{-y}}{r-1}\Bigl(r e^{-(r-1)y}-1\Bigr).
\]
Moreover,
\[
(r-1)y\le\frac{3\ell}{2}\cdot\frac{2}{3\ell}=1,
\]
so $e^{-(r-1)y}\ge e^{-1}$. Also, since $\ell\ge1$ and
$y\in A_{\ell}$, we have $y < 1$ and hence $e^{-y}\ge e^{-1}$. Therefore
\[
e^{-y}-p_{r}(x,y)\ge e^{-1}\cdot\frac{r e^{-1}-1}{r-1}.
\]
Since $r\in I_{\ell}$ and $\ell\ge\ell_{1}(x)\ge 4e$, we have
$r\ge\ell/2\ge 2e$, and therefore
\[
\frac{r e^{-1}-1}{r-1}\ge \frac{1}{2e},
\]
because
\[
\frac{r}{e}-1-\frac{r-1}{2e}=\frac{r-2e+1}{2e}\ge0.
\]
Hence
\[
e^{-y}-p_{r}(x,y)\ge \frac{1}{2e^{2}}=:C,
\qquad \ell\ge\ell_{1}(x),\ r\in I_{\ell},\ y\in A_{\ell}.
\]
Now let 
\[
p'_{\ell}(x,y):=\mathbb{E}\!\left[p_{C_{\ell}}(x,y)\right]
\]
denote the absolutely continuous part of the subordinated kernel.
Then for $\ell\ge\ell_{1}(x)$ and $y\in A_{\ell}$, 
\[
e^{-y}-p'_{\ell}(x,y)=\mathbb{E}\!\left[e^{-y}-p_{C_{\ell}}(x,y)\right]\ge C \mathbb{P}(C_{\ell}\in I_{\ell}).
\]
Therefore 
\begin{align*}
d_{\mathrm{TV}}\!\bigl(P'_{\ell}(x,\cdot),\pi\bigr)
&\ge \pi(A_{\ell})-P'_{\ell}(x,A_{\ell})\\
&=\int_{A_{\ell}}\bigl(\pi(y)-p'_{\ell}(x,y)\bigr) dy\\
&\ge\frac{2C}{3\ell} \mathbb{P}(C_{\ell}\in I_{\ell}).
\end{align*}
It remains to bound $\mathbb{P}(C_{\ell}\in I_{\ell})$ from below. By Chebyshev's inequality, 
\[
\mathbb{P}(C_{\ell}\notin I_{\ell})=\mathbb{P}\!\left(|C_{\ell}-\ell|>\frac{\ell}{2}\right)\le\frac{4\operatorname{Var}(C_{\ell})}{\ell^{2}}=\frac{4}{\ell}.
\]
Hence $\mathbb{P}(C_{\ell}\in I_{\ell})\to1$ and in particular there
exists $\ell_{3}$ such that 
\[
\mathbb{P}(C_{\ell}\in I_{\ell})\ge\frac{1}{2},\qquad\ell\ge\ell_{3}.
\]
Combining the last three displays yields 
\[
d_{\mathrm{TV}}\!\bigl(P'_{\ell}(x,\cdot),\pi\bigr)\ge\frac{C}{\ell}
\]
for some $C>0$ and all sufficiently large $\ell$.
\end{proof}

\begin{rem}[Implication for $n>2$]
Let $\mathcal{T}^{(n)}_{\ell}$ denote the local genealogy on $n$
labeled leaves, with stationary law $\Pi_{n}$. For any pair $(i,j)$,
the pairwise TMRCA satisfies
\[
\tau_{ij}(\ell)=g_{ij}(\mathcal{T}^{(n)}_{\ell})
\]
for a measurable map $g_{ij}$. Hence the data processing inequality gives
\[
d_{\mathrm{TV}}\!\bigl(\mathcal{L}(\tau_{ij}(\ell)\mid\mathcal{T}^{(n)}_{0}),\mathcal{L}_{\Pi_{n}}(\tau_{ij})\bigr)\le d_{\mathrm{TV}}\!\bigl(\mathcal{L}(\mathcal{T}^{(n)}_{\ell}\mid\mathcal{T}^{(n)}_{0}),\Pi_{n}\bigr),
\]
where $\mathcal{L}_{\Pi_{n}}(\tau_{ij})$ denotes the law of the pairwise TMRCA under the stationary distribution $\Pi_{n}$.
Since the left-hand side decays as $\asymp 1/\ell$ in the pairwise case,
we obtain the lower bound
\[
d_{\mathrm{TV}}\!\bigl(\mathcal{L}(\mathcal{T}^{(n)}_{\ell}\mid\mathcal{T}^{(n)}_{0}),\Pi_{n}\bigr)=\Omega\!\left(\frac{1}{\ell}\right).
\]
In particular, the full genealogy cannot converge to stationarity
faster than any fixed pairwise TMRCA. The same conclusion applies
to any statistic that depends on a fixed subset of leaves.
\end{rem}

\section{Discussion}
\label{sec:discussion}

In this paper, we studied the rate of forgetting for the pairwise
SMC and SMC' processes. For the embedded jump chains, we obtained explicit
geometric ergodicity bounds. For the continuous processes indexed by
genomic distance, by contrast, we showed that the total variation
distance from stationarity decays only as $1/\ell$.
The fact that correlations along the genome decay at a rate inversely
proportional to genetic distance is well known in population genetics.
Our contribution is to give a rigorous treatment of this phenomenon,
and to strengthen it by showing that this rate is sharp and that it
holds for total variation distance, not just for correlation decay.

These results have some practical consequences. Any procedure
that treats genealogies at separated loci as approximately independent must
respect the slow $1/\ell$ rate of decorrelation. Conversely, methods which sample based on
\emph{tree} distance (e.g.,~the number of recombination events separating
two loci) rather than genetic distance may be able to achieve faster
decorrelation. The caveat is that tree distance is not directly observable
and must be estimated from data, which introduces additional error.
Additionally, even when $n=2$, a randomly sampled \emph{tree} (as opposed to
the tree at a randomly sampled \emph{position}) does not
constitute a sample from the stationary distribution. For $n=2$, debiasing
the tree is straightforward, but for $n>2$ we are not aware of any existing procedure to do so.

The most obvious extension of these results is to the case $n>2$, where the state
space is more complicated and explicit formulas are not available.
The remark above shows that the full tree cannot converge faster
than $1/\ell$, since every pairwise TMRCA is a measurable function
of the tree. We conjecture that this lower bound is sharp, i.e.,
\[
d_{\mathrm{TV}}\!\bigl(\mathcal{L}(\mathcal{T}^{(n)}_{\ell}\mid\mathcal{T}^{(n)}_{0}),\Pi_{n}\bigr)\asymp\frac{1}{\ell},
\]
with constants depending on $n$. A proof would likely require quantitative
control of dependence between different pairwise coalescence times. Obtaining those is left to future work.

\section*{Acknowledgments}
This research was supported in part by the National Institute of General Medical Sciences of the NIH under award number R35GM151145.
The content is solely the responsibility of the authors and does not necessarily represent the official views of the NIH.
\printbibliography

\appendix

\section{Laplace transform and transition kernel of the SMC process}

In this section we derive the transition kernel of the SMC process by solving a PDE for the Laplace transform of the transition kernel. 

\begin{thm}
\label{thm:smc_density}The SMC process conditioned on $Y_{0}=x>0$
has Laplace transform $\phi(\ell,\lambda):=\mathbb{E}_{x}e^{-\lambda Y_{\ell}}$ given by
\begin{equation}
\phi(\ell,\lambda)=1-\frac{\lambda(1+\lambda+\ell)}{(1+\lambda)(\lambda+\ell)}\left(1-e^{-(\lambda+\ell)x}\right),
\quad\ell\ge0,\lambda>0,
\label{eq:phi}
\end{equation}
\end{thm}
\begin{proof}
Fix $\lambda>0$ and set $f_{\lambda}(y)=e^{-\lambda y}$. We have
\begin{align*}
\mathbb{E}[e^{-\lambda Z}] & =\int^{\infty}_{0}e^{-\lambda z}e^{-z} dz=\frac{1}{1+\lambda},\\
\mathbb{E}[e^{-\lambda Uy}] & =\int^{1}_{0}e^{-\lambda uy} du=\frac{1-e^{-\lambda y}}{\lambda y}.
\end{align*}
so that 
\begin{equation}
(Jf_{\lambda})(y)=\mathbb{E}\bigl[e^{-\lambda(Uy+Z)}\bigr]=\mathbb{E}\bigl[e^{-\lambda Uy}\bigr]\;\mathbb{E}\bigl[e^{-\lambda Z}\bigr]=\frac{1-e^{-\lambda y}}{\lambda y(1+\lambda)}.\label{eq:Jexp}
\end{equation}
Substituting (\ref{eq:Jexp}) into (\ref{eq:generator}) gives 
\begin{equation}
(Af_{\lambda})(y)=\frac{1-e^{-\lambda y}}{\lambda(1+\lambda)}-ye^{-\lambda y}.\label{eq:Lexp}
\end{equation}
By Dynkin's formula,
\[
\frac{\partial}{\partial\ell} \mathbb{E}_{x}[f(Y_{\ell})]=\mathbb{E}_{x}[(\mathcal{A}f)(Y_{\ell})].
\]
Applying this with $f=f_{\lambda}$ and using (\ref{eq:Lexp}), 
\begin{equation}
\frac{\partial}{\partial\ell}\phi(\ell,\lambda)=\frac{1}{\lambda(1+\lambda)}\mathbb{E}_{x}[1-e^{-\lambda Y_{\ell}}]-\mathbb{E}_{x}[Y_{\ell}e^{-\lambda Y_{\ell}}].\label{eq:phi-ell-deriv}
\end{equation}
The first expectation equals $1-\phi(\ell,\lambda)$. Additionally,
\begin{equation}
\mathbb{E}_{x}[Y_{\ell}e^{-\lambda Y_{\ell}}]=-\frac{\partial}{\partial\lambda}\mathbb{E}_{x}[e^{-\lambda Y_{\ell}}]=-\frac{\partial}{\partial\lambda}\phi(\ell,\lambda)\label{eq:eq:phi-lam-deriv}
\end{equation}
where the interchange of differentiation and integration is justified
by the bounded convergence theorem and the fact that $xe^{-\lambda x}\le1/(\lambda e)$.
Combining (\ref{eq:phi-ell-deriv}) and (\ref{eq:eq:phi-lam-deriv})
we obtain the transport equation 
\begin{equation}
\frac{\partial}{\partial\ell}\phi(\ell,\lambda)-\frac{\partial}{\partial\lambda}\phi(\ell,\lambda)=\frac{1}{\lambda(1+\lambda)}\bigl(1-\phi(\ell,\lambda)\bigr),\qquad\ell\ge0,\ \lambda>0,\label{eq:pde-phi}
\end{equation}
with initial condition $\phi(0,\lambda)=e^{-\lambda x}.$

To solve the PDE, we employ the method of characteristics. Firstly
change variables $\psi(\ell,\lambda):=1-\phi(\ell,\lambda)$ so that
(\ref{eq:pde-phi}) becomes 
\begin{equation}
\frac{\partial}{\partial\ell}\psi(\ell,\lambda)-\frac{\partial}{\partial\lambda}\psi(\ell,\lambda)=-\frac{1}{\lambda(1+\lambda)}\psi(\ell,\lambda),\qquad\psi(0,\lambda)=1-e^{-\lambda x}.\label{eq:pde-psi}
\end{equation}
Fix $(\ell,\lambda)$ and define for $s\in[0,\ell]$ the curves
\[
\ell(s)=s,\qquad\lambda(s)=\lambda+(\ell-s).
\]
Then $\ell(0)=0$, $\lambda(0)=\lambda+\ell$, and $\ell(\ell)=\ell$,
$\lambda(\ell)=\lambda$. By the chain rule, 
\[
\frac{d}{ds}\psi(\ell(s),\lambda(s))=\frac{\partial\psi}{\partial\ell}(\ell(s),\lambda(s))-\frac{\partial\psi}{\partial\lambda}(\ell(s),\lambda(s))
\]
so that by (\ref{eq:pde-psi}), 
\begin{equation}
\frac{d}{ds}\psi(\ell(s),\lambda(s))=-\frac{1}{\lambda(s)(1+\lambda(s))}\psi(\ell(s),\lambda(s)).\label{eq:ode}
\end{equation}
Since $\lambda>0\implies0<\phi<1\implies0<\psi<1$, we may divide
by $\psi$ and take logs: 
\[
\frac{d}{ds}\log\psi(\ell(s),\lambda(s))=-\frac{1}{\lambda(s)(1+\lambda(s))}.
\]
Integrating from $s=0$ to $s=\ell$ yields 
\[
\log\psi(\ell,\lambda)-\log\psi(0,\lambda+\ell)=-\int^{\ell}_{0}\frac{1}{\lambda(s)(1+\lambda(s))} ds.
\]
Now substitute $u=\lambda(s)=\lambda+\ell-s$, so $du=-ds$.
The limits become $u=\lambda+\ell$ at $s=0$ and $u=\lambda$
at $s=\ell$, and therefore 
\[
-\int^{\ell}_{0}\frac{1}{\lambda(s)(1+\lambda(s))} ds=\int^{\lambda}_{\lambda+\ell}\frac{1}{u(1+u)} du=\log\left(\frac{\lambda(1+\lambda+\ell)}{(1+\lambda)(\lambda+\ell)}\right).
\]
Exponentiating and using $\psi(0,z)=1-e^{-zx}$, we obtain equation
(\ref{eq:phi}).
\end{proof}

Inverting the Laplace transform, we can recover the conditional law
of $Y_{\ell}$ given in equations (\ref{eq:P_ell}) and (\ref{eq:p_c}).
Define 
\[
\phi_{\ell}(\lambda)=\phi(\ell,\lambda)=e^{-(\lambda+\ell)x}+\frac{\ell}{(1+\lambda)(\lambda+\ell)}\Bigl(1-e^{-(\lambda+\ell)x}\Bigr).
\]
The factor $e^{-(\lambda+\ell)x}=e^{-\ell x}e^{-\lambda x}$ implies an atom
at $x$ of weight $e^{-\ell x}$. For the absolutely continuous part,
let $F(\lambda)=\frac{\ell}{(1+\lambda)(\lambda+\ell)}=\mathcal{L}\{h(y)\}$
be the Laplace transform of a density $h$ to be determined later.
By the shift rule, 
\[
\mathcal{L}^{-1}\left\{ e^{-(\lambda+\ell)x}F(\lambda)\right\} =e^{-\ell x}h(y-x),\quad y>x.
\]
By partial fraction expansion, 
\[
\frac{1}{(1+\lambda)(\lambda+\ell)}=\frac{1}{\ell-1}\left(\frac{1}{\lambda+1}-\frac{1}{\lambda+\ell}\right)
\]
assuming $\ell\neq1$. Hence
\[
\mathcal{L}^{-1}\left\{ F(\lambda)\right\} =h(y)=\frac{\ell}{\ell-1}\bigl(e^{-y}-e^{-\ell y}\bigr),\quad y>0.
\]
Combining these expressions and simplifying yields the claim for $\ell\ne1.$
The $\ell=1$ case follows by taking limits.

\end{document}